\documentclass{amsart}

\usepackage{geometry}
\usepackage[english]{babel}

\usepackage{latexsym,amsfonts}
\usepackage{amssymb}
\usepackage{verbatim}
\usepackage{mathrsfs} 
\usepackage{amsaddr}

\def\tr{\mathrm{tr}}
\def\rk{\mathrm{rk}}

\def\GL{\mathrm{GL}}    
\def\Mat{\mathrm{Mat}}   
\def\I{\mathrm{I}}

\def\PSU{\mathrm{PSU}}
\def\PSL{\mathrm{PSL}}

\def\PSp{\mathrm{PSp}}

\def\SL{\mathrm{SL}}  
\def\SU{\mathrm{SU}}
\def\U{\mathrm{U}}
\def\Sp{\mathrm{Sp}}
\def\SO{\mathrm{SO}}

\def\Sym{\mathrm{Sym}}
\def\Alt{\mathrm{Alt}}
\def\Imm{\mathrm{Im}}

\def\P{{\rm P}}

\def\equad{\quad \textrm{ and } \quad}

\newcommand{\F}{\mathbb{F}}   
\newcommand{\K}{\mathbb{K}}   

\newtheorem{theorem}{Theorem}[section] 
\newtheorem{lemma}[theorem]{Lemma}     
\newtheorem{corollary}[theorem]{Corollary}

\theoremstyle{definition}
\newtheorem{remark}[theorem]{Remark}

\numberwithin{equation}{section}

\begin{document}

\title{The $(2,3)$-generation of the finite $8$-dimensional orthogonal groups}

\author{M.A. Pellegrini and M.C. Tamburini Bellani}
\email{marcoantonio.pellegrini@unicatt.it}
\email{mariaclara.tamburini@gmail.com}

\address{Dipartimento di Matematica e Fisica, Universit\`a Cattolica del Sacro Cuore,\\
Via della Garzetta 48, 25133 Brescia, Italy}

\begin{abstract}
We construct $(2,3)$-generators for the finite $8$-dimensional orthogonal groups,
proving the following results: the groups $\Omega_8^+(q)$ and $\P\Omega_8^+(q)$ are $(2,3)$-generated if and only
if $q\geq 4$;
the groups  $\Omega_8^-(q)$ and $\P\Omega_8^-(q)$ are $(2,3)$-generated for all $q\geq 2$.
\end{abstract}


\maketitle

\section{Introduction}

A group is said to be $(2,3)$-generated if it can be generated by an involution and an element of order
$3$. By a famous result of Liebeck and Shalev \cite{LS}, the finite classical simple groups are
$(2,3)$-generated, apart from the two infinite families $\PSp_4(q)$ with $q=2^f,3^f$, and a finite list 
$\mathcal{L}$ of exceptions. 
The problem of determining which are exactly the members of this list $\mathcal{L}$ requires a lot of detailed analysis, 
especially in small dimensions. Thanks to the efforts of many authors (see in particular \cite{SL12,Unit,Sp} and the references within)
it remains  open only for the orthogonal groups. 

According to \cite{Ischia}, the following finite simple classical groups are not $(2,3)$-generated:
\begin{itemize}
\item $\PSp_4(q)\cong \P\Omega_5(q)$ with $q=2^f, 3^f$;
\item $\PSL_2(9)\cong \P\Omega_3(9)\cong \P\Omega_4^-(3)$;
\item $\PSL_3(4)$;
\item $\PSL_4(2)\cong \P\Omega_6^+(2)$;
\item $\PSU_3(q^2)$ with $q=3,5$;
\item $\PSU_4(q^2)\cong \P\Omega_6^-(q)$ with $q=2,3$;
\item $\PSU_5(2^2)$;
\item $\P\Omega_8^+(q)$ with $q=2,3$.
\end{itemize}
The isomorphisms are well known, see \cite{Atlas}. From this list, it turns out that most exceptions are realized by orthogonal groups.
Also for this reason, they are the most difficult case to be studied with respect to the $(2,3)$-generation problem.
The two exceptions $\P\Omega_8^+(2)$ and $\P\Omega_8^+(3)$ have been found by Vsemirnov \cite{Max}, while
in \cite{Sp6} it was proved that the groups $\Omega_7(q)$ are $(2,3)$-generated for all odd $q$.
Note that, for $q$ even, $\Omega_7(q)\cong \Sp_6(q)$. 
It is worth mentioning that all the known  exceptions are $(2,5)$-generated, apart from $\PSU_3(3^2)$ which 
is $(2,7)$-generated. So they are all $(2,r)$-generated, for some small prime $r$.

Here, given a finite $8$-dimensional orthogonal group, we construct plenty of $(2,3)$-generating pairs $x,y$,
where $y$ is fixed and $x$ depends on one or two parameters. In particular, we prove:

\begin{theorem}\label{main}
The groups $\Omega_8^+(q)$ and $\P\Omega_8^+(q)$ are $(2,3)$-generated if and only if $q\geq 4$,
while the groups  $\Omega_8^-(q)$ and $\P\Omega_8^-(q)$  are $(2,3)$-generated for all $q\geq 2$.
\end{theorem}

\section{Notation and preliminary results}

Denote by $\F$ the algebraic closure of the field $\F_p$  of prime order $p$.
For every $q=p^f$, the group $\GL_{8}(q)$  acts on the left on the space $\F_q^{8}$,
whose canonical basis is denoted by $\{e_1,e_2,e_3,e_4,e_{-1},e_{-2},e_{-3},e_{-4}\}$.
Denote by  $\omega$ a primitive cubic root of $1$ in $\F$.

Given $g\in \Mat_{m,n}(q)$, $m\geq n$, we denote by $g_{(i_1,\ldots,i_n)}$ the square submatrix of $g$ 
consisting of rows $i_1,\ldots,i_n$.
Moreover, if $\sigma$ is a field automorphism, we set $g^\sigma$ for the matrix obtained from $g$ applying
$\sigma$ to its entries.

The lists of maximal subgroups of $\Omega_8^+(q)$, obtained by Kleidman \cite{K}, and of $\Omega_8^-(q)$
are  described in \cite[Tables 8.50--8.53]{Ho}: this will be our main reference for the subdivision in 
classes $\mathcal{C}_1,\ldots,\mathcal{C}_6,\mathcal{S}$, and for the structure of these subgroups.

Let $K\leq \GL_n(\K)$ and $U$ be a $K$-invariant subspace of $V=\K^n$.
Then $U$ has a $K^T$-invariant complement $\overline{U}$.
Assume further that $K$ preserves a non singular form $J$, namely that $J k J^{-1} = k^{-T}$ for all
$k \in K$. It follows that $J\overline{U}$ is a $K$-invariant subspace of dimension $n-\dim(U)$.
In particular this applies to $K\leq \Omega^{\pm}_{2n}(q)$, in which case we may always assume $\dim(U)\leq n$.

Given an eigenvalue $\lambda$ of a matrix $g \in \GL_n(\K)$, with corresponding eigenspace $V_{\lambda}(g)$, 
we define the following $g$-invariant subspace: $E_\lambda(g)=V_\lambda(g)\cap \Imm(g-\I_n)$.

For any finite group $G$, let $\Upsilon(G)$ be the set of the prime divisors of $|G|$.
For simplicity, if $g \in G$,  we write $\Upsilon(g)$ for $\Upsilon(\langle g \rangle)$.

\section{Generators for $\Omega_8^+(q)$ and the triality automorphism}\label{gen+}

Since $\Omega_8^+(2)$ and $\Omega_8^+(3)$ are not $(2,3)$-generated, in this section
and in Section \ref{gen8+} we assume $q\neq 2,3$.
Let $a\in \F_q^*$ be such that $\F_p[a]=\F_q$. 
We define our generators $x=x(a), y$ as follows:
$$
x=\begin{pmatrix}
0&0&0&0&0&1&0&0\\
0&0&0&0&1& (a+1)^2 & (a+1)^2&a\\
0&0&0&0&0 & (a+1)^2 & (a+1)^2 &a\\
0&0&0&0&0&a&a&1\\
0&1&1&0&0&0&0&0\\
1&0&0&0&0&0&0&0\\
1&0&1&a&0&0&0&0\\
0&0&a& (a+1)^2 &0&0&0&0
\end{pmatrix} \quad \textrm{if } p=2,$$
$$x=\begin{pmatrix}
0&-1&0&0&0&0&0&0\\
-1&0&0&0&0&0&0&0\\
1&-1&-1&-a&0&0&0&0\\
0&0&0&1&0&0&0&0\\
0&0&0&1&0&-1&1&0\\
0&0&0&-1&-1&0&-1&0\\
0&0&0&-2&0&0&-1&0\\
1&-1&-2&-2a&0&0&-a&1
\end{pmatrix} \quad \textrm{if } p>2,$$ 
$$y=
\begin{pmatrix}
1&0&0&0&0&0&0&0\\
0&0&0&0&0&0&0&1\\
0&0&0&0&0&1&0&0\\
0&0&1&0&0&0&0&0\\
0&0&0&0&1&0&0&0\\
0&0&0&1&0&0&0&0\\
0&1&0&0&0&0&0&0\\
0&0&0&0&0&0&1&0
\end{pmatrix}.$$
Then, $x$ and $y$ have respective orders $2$ and $3$. Moreover they belong to the group 
$\SO^+_8(q)$ defined by the quadratic form 
$Q\left(\sum\limits_{i=1}^4 (\alpha_i e_i+\alpha_{-i}e_{-i})\right)= \sum\limits_{i=1}^4 \alpha_i \alpha_{-i}$.

The group $\P\Omega^+_8(q)$ has an outer automorphism $\tau$ of order $3$, called triality, which arises from a 
symmetry of order $3$ of the Dynkin diagram of type $D_4$. It can also be read from the spin representation of 
the Clifford group (e.g., see \cite[pages 78--81]{W}). We describe $\tau$ in terms of the 
Steinberg generators of $\Omega^+_8(q)$ (see \cite[page 185]{C}), namely:
$$\begin{array}{rclcrcl}
x_{1,2}(\alpha) & = & \I_8+\alpha (E_{1,2}-E_{-2,-1}),& \quad & x_{2,1}(\alpha) &= & \left(x_{1,2}(\alpha)\right)^T,\\
x_{2,3}(\alpha) & = & \I_8+\alpha (E_{2,3}-E_{-3,-2}), & & x_{3,2}(\alpha) & = &\left(x_{2,3}(\alpha)\right)^T,\\
x_{3,4}(\alpha) & = & \I_8+\alpha (E_{3,4}-E_{-4,-3}),& &  x_{4,3}(\alpha)& =&\left(x_{3,4}(\alpha)\right)^T,\\
x_{1,-2}(\alpha) & = & \I_8+\alpha (E_{1,-2}-E_{2,-1}),  & & x_{-2,1}(\alpha)&= &\left(x_{1,-2}(\alpha)\right)^T.\\
\end{array}$$
Then $\tau$ is defined by the following permutation action \cite[Proposition 12.2.3]{C}:
\begin{equation}\label{triality}
\left(x_{1,2}(\alpha), x_{3,4}(\alpha), x_{-2,1}(\alpha)\right)\left(x_{2,1}(\alpha), x_{4,3}(\alpha), x_{1,-2}(\alpha)\right).
\end{equation}

We need $\tau$ mainly because it moves certain  maximal subgroups from class $\mathcal{C}_1$  to some other  class  $\mathcal{C}_j$.
So, to exclude that $H\leq M\in \mathcal{C}_j$ we show that $\tau^i(H)$, $i=1,2$, is absolutely irreducible,
hence it cannot be contained in any maximal subgroup of class $\mathcal{C}_1$.

In order to calculate $\tau(x)$ and $\tau(y)$, we need to express $x,y$ as products of the above generators,
proving in particular that $H\leq \Omega_8^+(q)$.
To this purpose it is useful to consider the factorization $\P\Omega_8^+(q)=BNB$,
where $B$ is the Borel subgroup and $N$ the monomial subgroup.
As a special case of the formula $[\I_n+\alpha E_{i,j}, \I_n+\beta E_{j,k}]= \I_n+\alpha\beta E_{i,k}$ for $i,j,k$ distinct, we get: 
$$x_{2,4}(\alpha) = [x_{2,3}(\alpha), x_{3,4}(1)],\quad
x_{1,3}(\alpha) = [x_{1,2}(\alpha),x_{2,3}(1)],\quad
x_{1,4}(\alpha) = [x_{1,3}(\alpha),x_{3,4}(1)].$$
Now, set
$$\begin{array}{rclcrcl}
\pi_{1,2} & = & x_{2,1}(-1)\;x_{1,2}(1)\;x_{2,1}(-1),&\; & \pi_{2,3} & =&  x_{3,2}(-1)\;x_{2,3}(1)\;x_{3,2}(-1), \\
\pi_{3,4} & = & x_{4,3}(-1)\;x_{3,4}(1)\;x_{4,3}(-1),& & \pi_{1,-2} & = & x_{-2,1}(-1)\;x_{1,-2}(1) \;x_{-2,1}(-1), \\
\pi_{1,3} & = & \pi_{1,2}^{-1}\;\pi_{2,3}\;\pi_{1,2},& & \pi_{2,4} & = & \pi_{2,3}^{-1}\; \pi_{3,4}\; \pi_{2,3}, \\
\end{array}$$
and then 
$$\rho_{1,2} =\pi_{1,2}\;\pi_{1,-2},\quad
\rho_{3,4} =  \pi_{1,3}\;\pi_{2,4}\;\rho_{1,2}\;\pi_{1,3}\;\pi_{2,4},\quad
J =  \rho_{1,2}\;\rho_{3,4}.$$
Finally, we define
$$\begin{array}{rclcrcl}
x_{-1,3}(\alpha) &  = & (x_{-2,1}(\alpha))^{\pi_{2,3}}, &  & x_{-1,4}(\alpha) & =& (x_{-1,3}(1))^{\pi_{3,4}},\\
x_{-2,4}(\alpha) & =&  (x_{-1,4}(\alpha))^{\pi_{1,2}}, & & 
x_{-3,4}(\alpha) & =& (x_{-2,4}(\alpha))^{\pi_{2,3}}.\\
\end{array}$$
Now $H\leq \Omega_8^+(q)$, since we can write 
$$y=\rho_{1,2}\;\pi_{2,3}\;\rho_{1,2}\; \pi_{3,4}^3$$
and 
$$x=\left\{\begin{array}{ll}
  x_{2,3}(1) \;x_{3,4}(a)\; \pi_{1,2}\; x_{4,3}(a)\; x_{3,2}(1) \;J & \textrm{ if } p=2,\\[5pt]
   x_{-1,4}(-1)\; x_{-2,4}(1)\; x_{-3,4}(2)\; x_{1,2}(1)\; x_{2,3}(-1) \;
x_{2,4}(-a) \cdot\\
x_{1,3}(1)\; x_{1,4}(a)\; \pi_{2,3}^2 \;\pi_{1,3}^{-1}\;x_{1,2}(-1)\; x_{2,3}(1) \; x_{1,4}(-a)  & \textrm{ if } p>2.
    \end{array}\right.$$

Applying \eqref{triality} we get:
$$\tau(y)=
\begin{pmatrix}
0&0&0&-1&0&0&0&0\\
0&0&0&0&1&0&0&0\\
0&0&-1&0&0&0&0&0\\
0&0&0&0&0&1&0&0\\
0&0&0&0&0&0&0&-1\\
1&0&0&0&0&0&0&0\\
0&0&0&0&0&0&-1&0\\
0&1&0&0&0&0&0&0
\end{pmatrix}, \; \tau^2(y)=\begin{pmatrix}
0&0&0&0&0&0&0&1\\
-1&0&0&0&0&0&0&0\\
0&0&-1&0&0&0&0&0\\
0&0&0&0&0&1&0&0\\
0&0&0&1&0&0&0&0\\
0&0&0&0&-1&0&0&0\\
0&0&0&0&0&0&-1&0\\
0&1&0&0&0&0&0&0
\end{pmatrix}.$$
Furthermore, $\tau(x)$ and $\tau^2(x)$ are, respectively, the matrices
$$
\begin{pmatrix}
0&a&a&0&1&0&0&0\\
a&0&0&0&0&1&0&1\\
0&0&0&0&0&0&0&1\\
0&0&0&0&0&1&1&0\\
a^2+1&0&0&0&0&a&0&0\\
0&a^2+1&a^2+1&0&a&0&0&0\\
0&a^2+1&a^2+1&1&a&0&0&0\\
0&0&1&0&0&0&0&0
\end{pmatrix},\;
\begin{pmatrix}
a&a^2+1&a^2+1&0&0&0&0&0\\
1&a&a&0&0&1&1&0\\
0&0&0&0&0&1&1&0\\
0&0&0&0&0&0&0&1\\
0&0&0&0&a&1&0&0\\
0&0&0&0&a^2+1&a&0&0\\
0&0&1&0&a^2+1&a&0&0\\
0&0&0&1&0&0&0&0
\end{pmatrix}$$
if $p=2$, and            
$$
\begin{pmatrix}
1&0&0&0&0&0&0&0\\
0&-1&0&0&0&0&0&0\\
0&1&0&1&0&0&0&0\\
0&1&1&0&0&0&0&0\\
0&-a&0&0&1&0&0&0\\
-a&0&-1&1&0&-1&1&1\\
0&-1&-2&0&0&0&0&1\\
0&1&0&2&0&0&1&0
\end{pmatrix},
\quad 
 \begin{pmatrix}
0&0&0&-1&-a&1&-1&0\\
0&0&0&0&1&0&0&0\\
0&-1&-1&0&-1&0&0&0\\
0&0&0&1&0&0&0&0\\
0&1&0&0&0&0&0&0\\
1&a&0&-1&0&0&-1&0\\
0&0&0&-2&0&0&-1&0\\
0&-1&-2&0&-1&0&0&1
\end{pmatrix}$$
if $p>2$.

\begin{lemma}\label{irr+2}
If $p=2$ the groups $H$, $\tau(H)$ and $\tau^2(H)$ are absolutely irreducible. 		
\end{lemma}

\begin{proof}
We begin by proving the absolute irreducibility of $H$.
The eigenspace $V_\omega(y)$ is generated by $v_i=v_i(\omega)$, $i=1,2$, where:
$$v_1=e_2 + \omega^{-1} e_{-3} +  \omega e_{-4} \equad
  v_2=e_3 + \omega^{-1} e_{4} +  \omega e_{-2}.$$
Let $\left\{0\right\}\neq U$ be an $H$-invariant subspace of $\F^8$ and $\overline U$ be an $H^T$-invariant complement of $U$. 
We may assume $\dim U\leq 4$. 
If $U\cap V_\omega(y) =\left\{0\right\}$, then $V_\omega(y^T)\leq \overline U$. From $y^T=y^{-1}$ it follows that 
$V_\omega(y^T)=\langle v_1^-, v_2^- \rangle$, where $v_1^-=v_1(\omega^{-1})$ and $v_2^-=v_2(\omega^{-1})$. The matrix $M$ whose columns are 
$$v_1^-,\; x^T v_1^-,\;  (xy)^T v_1^-,\; (xyx)^Tv_1^-,\;  v_2^-, \;  x^Tv_2^-,\; (xy)^T v_2^-,\;  (xyx)^Tv_2^-$$
has determinant $(a+1)^4 (a^2+a+\omega)^4$, which is nonzero unless $a^2=a+\omega$.
In this case, take the matrix $\overline M$ obtained by replacing the last two columns of $M$ with
$((xy)^2)^T v_1^-$ and $((xy)^2x)^T v_1^-$. Then $\det(\overline M)=(a+1)^3\neq 0$.
So, in any case, we obtain the contradiction $U=\left\{0\right\}$. It follows that, for some $(b_1,b_2)\neq (0,0)$,
$v=b_1v_1+b_2v_2\in U$.

Let $N$ be the matrix whose columns are the $v,\; xv,\; yxv,\; y^2xv,\; (yx)^2v$.
If $v=v_2$,  then $\det(N_{(1, 2, 3, 4, 8)})=(a+1)^6$ and hence $N$ has rank $5$, in contrast with $\dim U\leq 4$.
If $v=v_1$, again $v_2\in U$ as 
$$(a^2+a)v_2=(a^2 + \omega^2a + 1) v_1+ \omega xv_1+ yxv_1+ \omega^2y^2xv_1.$$
So, assume $v=b_1v_1+v_2$ with $b_1\neq 0$, and set
$A=\left((a+1)b_1+\omega\right)\left(a b_1+\omega(a+1)\right)$.
If $A\neq 0$, again $v_2\in U$ as 
$$Av_2=( (a^2 + \omega^2 a + 1) b_1  + \omega^2(a+1)) v+ \omega b_1 xv+b_1 yxv+ \omega^2b_1y^2xv.$$
If $A=0$, then $b_1\in \left\{\frac{\omega}{a+1}, \frac{\omega (a+1)}{a}\right\}$. In both cases, 
take the matrix $B$ whose columns are $v,\; xv,\; yxv,\; xyxv,\; (yx)^2 v$.
If $b_1=\frac{\omega}{a+1}$, then $\det(B_{(1,2,3,4,5)})=\frac{d_1}{(a+1)^5}$
and $\det(B_{(1,2,3,4,7)})=\frac{d_2}{(a+1)^5}$, 
where 
$$\begin{array}{rcl}
d_1 & =& \omega^2 a^{13} + a^{12} + \omega^2 a^{10} + \omega a^9 + \omega a^8 + a^7 + \omega^2 a^6 + a^5 +
    \omega^2 a^4 + \omega a^3 + \omega a^2 + a + \omega^2,\\
d_2 & =& \omega^2 a^{15} + \omega a^{14} + \omega a^{13} + \omega^2 a^{12} + a^{10} + a^8 + \omega a^7 + a^6 + a^5 +
    \omega^2 a^4 +  a^3 + \omega a^2 +\\
    && a + \omega.
\end{array}$$
If $b_1=\frac{\omega(a+1)}{a}$, then $\det(B_{(1,2,3,4,5)})=\frac{d_1}{a^5}$
and $\det(B_{(1,2,3,4,7)})=\frac{d_2}{a^5}$, 
where 
$$\begin{array}{rcl}
d_1 & =& a^{11} + \omega^2 a^{10} + a^9 + \omega^2 a^8 + \omega a^7 + \omega^2 a^6 + \omega a^5 + \omega a^4 +
    \omega a^3 + a + \omega,\\
d_2 & =& a^{13} + \omega^2 a^{12} + a^{10} + \omega^2 a^8 + \omega a^7 + \omega^2 a^6 + \omega a^4 + a^3 + \omega a + 1.
    \end{array}$$
In both cases, $d_1$ and $d_2$ are coprime, whence the contradiction $\rk(B)=5$.

Now, to prove the absolute irreducibility of the groups  $\tau(H),\tau^2(H)$, we simply adapt the previous argument.
For sake of brevity, we only describe the necessary modifications.

Consider $\tau(H)$ and write $x$ for $\tau(x)$ and $y$ for $\tau(y)$.
The eigenspace $V_\omega(y)$ is generated by $v_i=v_i(\omega)$, $i=1,2$, where
$$v_1=e_1+ \omega e_4 +\omega^{-1} e_{-2} \equad 
  v_2=e_2+ \omega e_{-1} +\omega^{-1} e_{-4} .$$
Then $\det(M)=\omega (a+1)^2 (a^2+a+\omega)^4$ and, assuming $a^2=a+\omega$, $\det(\overline M)=\omega^2 (a+1)^2\neq 0$.
If $v=v_2$,  then $\det(N_{(1, 2, 3, 5, 7)})=\omega  a (a+1)^3\neq 0$.
Furthermore,  $(a^2+a)v_2= \omega^2 v_1+ \omega xv_1+ yxv_1+  \omega^2y^2xv_1$
and $$Av_2= \omega^2 (a^2 + a+ b_1) v+ \omega b_1 xv+ b_1 yxv+ \omega^2 b_1y^2xv,$$
where
$A=\left((a+1)b_1+\omega a \right)\left(a b_1+\omega(a+1)\right)$.

If $b_1=\frac{\omega a }{a+1}$, then $\det(B_{(1,2,3,4,5)})=\frac{d_1}{(a+1)^5}$
and $\det(B_{(1,2,3,4,6)})=\frac{d_2}{(a+1)^5}$, 
where 
$$\begin{array}{rcl}
d_1 & =& a^{14} + \omega^2 a^{13} + a^{12} + a^{11} + \omega a^9 + \omega a^8 + \omega^2 a^7 + \omega^2 a^5 + a^4 +
    \omega^2 a^3 + a + \omega,\\
 d_2 & =& a^{15} + \omega a^{14} + \omega a^{13} + a^{12} + \omega a^{10} + \omega a^9 + \omega^2 a^7 + \omega^2 a^6 +
    \omega a^5 +  \omega^2 a^4 + \omega^2 a^3 +\\
    && \omega a + 1.
     \end{array}$$
If $b_1=\frac{\omega(a+1)}{a}$, then $\det(B_{(1,2,3,4,5)})=\frac{d_1}{a^4}$
and $\det(B_{(1,2,3,4,7)})=\frac{d_2}{a^4}$, 
where 
$$\begin{array}{rcl}
d_1 & =& a^{13} + \omega a^{12} + \omega a^{11} + \omega a^{10} + a^8 + a^7 + \omega a^5 + \omega^2 a^4 + a^3 +
    \omega^2 a^2 + a + \omega,\\
d_2 & =& a^{14} + \omega a^{13} + \omega^2 a^{12} + a^{11} + \omega a^{10} + a^8 + \omega a^7 + \omega^2 a^6 +
    \omega^2 a^4 + \omega a^3 +\omega^2 a^2 + \\
    && \omega^2 a + \omega.
\end{array}$$
In both cases, $d_1$ and $d_2$ are coprime, whence the contradiction $\rk(B)=5$.

Finally, consider $\tau^2(H)$ and write $x$ for $\tau^2(x)$ and $y$ for $\tau^2(y)$.
The eigenspace $V_\omega(y)$ is generated by $v_i=v_i(\omega)$, $i=1,2$, where
$$v_1=e_1 + \omega^{-1} e_2 +\omega  e_{-4}  \equad v_2=e_4 + \omega^{-1} e_{-1} + \omega e_{-2}.$$
Then $\det(M)= a^2 (a+1)^2 (a^2+a+\omega)^4$
and, assuming $a^2=a+\omega$, $\det(\overline M)= (a+1)^3 \neq 0$.
Take as $N$ the matrix having columns $v,\;xv,\;yxv,\; xy^2xv,\; (yx)^2v$.
If $v=v_2$,  then $\det(N_{(1, 2, 3, 4, 7)})=\omega  a^2(a+1)^4\neq 0$.
Furthermore,  $\omega v_2=(\omega a+1)^2 v_1+   xv_1 + \omega^2  yxv_1+  \omega  y^2x v_1$
and $$ (a^2 + \omega b_1) v_2= (\omega  a+1)^2 v+  xv + \omega^2 yx v+ \omega  y^2xv.$$
If $b_1=(\omega a)^2$,  
then $\det(N_{(1,2,3,4,5)})=d_1$ and $\det(N_{(1, 3, 4, 5, 6 )})=d_2$, 
where 
$$\begin{array}{rcl}
d_1 & =& a^{13} + a^{12} + a^{11} + a^{10} + a^9 + \omega a^7 + a^6 + a^3 + \omega^2 a + 1,\\
d_2 & =& \omega a^{12} + \omega a^{11} + a^{10} + \omega a^8 + a^6 + \omega a^5 + \omega a^3 + a^2 + \omega^2 a + \omega.
\end{array}$$
Since $d_1$ and $d_2$ are coprime,  we obtain the contradiction  $\rk(N)=5$.
\end{proof}

\begin{lemma}\label{irr+odd}
Let $p$ be odd. If $a\neq \pm 2$, the group $H$ is absolutely irreducible and 
contains an element $\eta$ whose order is divisible by $p$.
\end{lemma} 

\begin{proof}
The characteristic polynomial of $\eta=(xy^2)^2(xy)^2$ is $\chi_\eta(t)=(t-1)^4 \psi(t)$ with 
$$\psi(t)=t^4- (4a^3 + 12a^2 - 14)t^3 + (4a^4 - 8a^3 - 44a^2 + 51)t^2  - ( 4a^3 + 12a^2 - 14)t + 1.$$
By direct computations, if $a\neq -2$, the eigenspace $V_1(\eta)$ is generated by the vectors
$$ v_1=e_1+e_2+(1-a)e_{-2} +e_{-3} + a e_{-4} \equad v_2=e_{-1}-e_{-2}+e_{-4}.$$
Note that $\eta^x=\eta^{-1}$ and that $\left\langle v,xv\right\rangle=V_1(\eta)$
for all $0\neq v$ in $V_1(\eta)$ unless $v$ is either a multiple of $v_2$ or a multiple of $\tilde v$, where
$$\tilde v=   2(e_1 + e_2 + e_{-3} )  -a e_{-1}+(2-a)e_{-2}+ a e_{-4}.$$

Let $U$ be an $H$-invariant subspace of $\F^8$ and $\overline{U}$ be an $H^T$-invariant complement of $U$.
We may assume $\dim U\leq 4$. Suppose first  that $\eta_{|U}$ has the eigenvalue $1$.
By what observed above, we may assume that $U$ contains the vector $v\in \{v_2,\tilde v\}$.
Let $M$ be the matrix whose columns are $v,\;yv,\; y^2v,\; xyv,\; xy^2v$.
Then $\det(M_{(1,2,3,4,5)})=-a$ if $v=v_2$ and $\det(M_{(1, 2, 3, 6, 8 )})=2(a-2)^4$ if $v=\tilde v$.
So, in both cases, $M$ has rank $5$ under our assumptions. 

Now, suppose that $\eta_{|U}$ does not have the eigenvalue $1$. 
Then $\overline{U}$ contains $V_1(\eta^T)$ and, in particular, contains the eigenvector 
$u=e_1-e_2+e_4$.
The matrix $N_1$ whose columns are the vectors 
\begin{equation}\label{N12}
u,\; y^Tu, \;(y^2)^T u,\; (yx)^T u,\;  (y^2x)^T u,\; (yxy)^T u, 
\end{equation}
and the vectors $(yxy^2)^T u,\; ((yx)^2)^Tu$ has determinant $(a-2)d_1$, where $d_1= (a^2-a-3) (a^2-a+1)$. The
matrix $N_2$ whose columns are the vectors \eqref{N12} and 
$((yx)^3)^Tu, \; ((yx)^4)^T u$ has determinant $(a-2)d_2$, where $d_2=a^6 - 2 a^5 - 7 a^4 + 6 a^3 + 15 a^2 - 8 a + 6$.
If $a^2=a+3$, then $d_2=-16a\neq 0$; if $a^2=a-1$, then $d_2=16a^2\neq 0$.
We conclude that $\dim(\overline{U})=8$.

Finally, since $1$ is eigenvalue of $\eta$ having algebraic multiplicity at least $4$, but geometric multiplicity equals to $2$, 
we conclude that the order of $\eta$ is divisible by $p$.
\end{proof}

\begin{lemma}\label{irr+tau}
Let $p$ be odd. If $a\neq -2$, the group $\tau(H)$ is absolutely irreducible.
\end{lemma}

\begin{proof}
We write $x$ for $\tau(x)$ and $y$ for $-\tau(y)$.
We have that  $E=E_\omega(y)$ is generated by $v_i=v_i(\omega)$, $i=1,2$, where:
$$v_1(\omega)=e_1+\omega e_4-\omega^{-1} e_{-2} \equad v_2(\omega)=e_2-\omega e_{-1} -\omega^{-1} e_{-4}.$$
Let $\left\{0\right\}\neq U$ be a $\tau(H)$-invariant subspace of $\F^8$ and $\overline U$ an $\tau(H)^T$-invariant complement of $U$. 
If $U\cap E=\left\{0\right\}$, then $E_\omega(y^T)=E_{\omega^{-1}}(y)\leq \overline U$. 
Taking $v_1^-=v_1(\omega^{-1})$ and $v_2^-=v_2(\omega^{-1})$, 
let $M$ and $N$ be the matrices whose columns are, respectively 
$$v_1^-,\; x^T v_1^-,\; (xy)^T v_1^-,\;(xyx)^T v_1^-,\;((xy)^2)^T v_1^-,\;((xy)^2x)^T v_1^-,\;
  v_2^-,\; x^T v_2^-$$
  and 
$$v_1^-,\; x^T v_1^-,\; (xy)^T v_1^-,\;(xyx)^T v_1^-,\; 
  v_2^-,\; x^T v_2^-,\; (xy)^T v_2^-,\; (xyx)^T v_2^- .$$
Then $M$ has determinant $a (a+2) (a-\omega+2)^2 (a+\omega)^2$, which is nonzero unless $a\in \{-\omega, \omega-2\}$.
If $a=-\omega$, then $\det(N)=-2^6(a+2)$;
if $a=\omega - 2$, then $\det(N)= -2^4 (3\omega+1)$, which is nonzero unless $p=7$ and $\omega=2$,
which gives the contradiction $a=0$. In any case, we get the contradiction $U=\left\{0\right\}$.
 
Hence, there exists $0\neq u \in E\cap U$.
Let  $A$ be the matrix having columns 
$u,\; x u,\; yx u,\; xyx u,$ $yxy^2x u$.
If $u=v_1$, then $A$ has rank $5$, since $\det(A_{(1, 2, 3, 5, 7)})=-8\omega a$. Thus, we may assume $u=\lambda v_1+v_2$.
In this case,  $\det(A_{(2,3,5,7,8)})\neq 0$, except in the following cases:
$$\begin{array}{rrcrr}
(1): \;  & \lambda =\frac{a\omega-1}{2};& \quad & 
(2):\;  & \lambda^2 + \frac{(a + 1) (\omega^{2}-1 )}{2} \lambda - \frac{\omega^{2}}{2} =0.
  \end{array}$$
Let $B$ be the matrix whose columns are
$u,\; xu, \; yx u,\; (yx)^2 u,\;  (yx)^3 u$.\\
\textbf{Case (1).} 
If $p=3$, then  $\det(B_{(2, 3, 4, 5, 8 )})=-(a+2)^3\neq 0$.
So, assume $p\neq 3$. Then $\det(B_{(2, 3, 4, 5, 8 )})=
 \frac{(\omega^2-\omega) (a + 2)^3}{36}(3a + (5\omega + 1))^3$  is nonzero unless  $a=-\frac{5\omega+1}{3}$, in which case $p\neq 5$ and 
$\det(B_{(1, 2, 4, 6, 8)})=\frac{2^6\cdot 5}{3^4}\neq 0$. 
It follows that $\rk(B)=5$.\\
\textbf{Case (2).} We may assume $\lambda \neq \frac{a\omega-1}{2}$ by previous analysis, and clearly $\lambda \neq 0$.
If $p=3$, from (2) we get $\lambda=\pm \iota$, where $\iota$ has order $4$.
For $\lambda\in \{\iota, -\iota\}$, $\det(B_{(  1, 2, 6, 7, 8)})=0$ if and only if $a=\pm \iota$,
in which case $\det(B_{(  1,2,3,4,5)})\neq 0$.
So, we may assume $p\neq 3$ and set $a=\frac{2(1-\omega)}{3}\lambda -1+\frac{2+\omega}{3\lambda}$. 
Then,  
$\det(A_{(1,2,3,4,5)})=\frac{2 d_1}{3^2\lambda^2}$,
$\det(A_{(1,2,3,4,6)})=\frac{2 d_2}{3^3\lambda^2}$ and 
$\det(A_{( 1,2,3,4,7)})=\frac{2 d_3}{3^2\lambda^2}$, where $d_1,d_2,d_3$ can be viewed as polynomials in $\lambda$.
Setting $\delta=\gcd(d_2,d_3)$, we get
$\delta=c_2 d_2 + c_3 d_3$, where
$$\begin{array}{rcl}
4\delta & =& 4\lambda^6 +2 (2\omega + 1)\lambda^5 - (2\omega + 1)\lambda^4 
-(\omega - 1)\lambda^3 - (\omega + 5)\lambda^2 -  (5\omega + 1)\lambda + \omega,\\
144 c_2  & =& 2(7\omega + 11)\lambda^2 - (7\omega - 25)\lambda - 2(\omega + 5),\\
48 c_3 & =& -4(\omega +6)\lambda^4 -2 (9\omega + 14)\lambda^3 - (19\omega + 32)\lambda^2 
+ 8(2\omega -1)\lambda  + 4(2\omega + 5).
  \end{array}$$
Finally, $\gcd(d_1,\delta)=1$, whence the contradiction $\rk(A)=5$.
\end{proof}

\begin{lemma}\label{irr+tau2}
Let $p$ be odd. If $a\neq -2$, the group $\tau^2(H)$ is absolutely irreducible.
\end{lemma} 

\begin{proof}
We write $x$ for $\tau^2(x)$ and $y$ for $-\tau^2(y)$.
The characteristic polynomial of $\eta=[x,y]$ is $\chi_\eta(t)=(t+1)^4 \psi(t)$ with 
$$\psi(t)=t^4 + (a^2 - 4) t^3 - (a^2 - 6) t^2 + (a^2 - 4) t + 1.$$
By direct computations, the eigenspace $V_{-1}(\eta)$ is generated by the vectors
$$v_1 = ae_1 -e_2+ e_3 +a e_4 -e_{-1} -a e_{-3}+(a^2+a-1)e_{-4} \equad 
v_2= e_{-2}+e_{-3}.$$

Let $U$ be an $\tau^2(H)$-invariant subspace of $\F^8$ and  $\overline{U}$ be an $(\tau^2(H))^T$-invariant complement of $U$.
We may assume $\dim U\leq 4$. Suppose first  that $\eta_{|U}$ has the eigenvalue $-1$.
Let $M$ be the matrix whose columns are the vectors $u,\; yu,\; y^2u,\; xy^2u,\; yxy^2u$, where $u \in U$.
If $u=v_2\in U$, then $\det(M_{( 1, 3, 5, 6, 7 )})=a+2\neq 0$.
So, assume $u=v_1+\lambda v_2\in U$ for some $\lambda \in \F$.
Since $\frac{1}{2} ( u+xu) = v_1$, we obtain that $v_1 \in U$.
Taking $u=v_1$, we have $\det(M_{(1, 2, 3, 4, 5)})=-d_1$ and 
$\det(M_{(1, 2, 3, 5, 6)})=(a+2)(a+\omega) (a+\omega^2) d_2$, where
$$\begin{array}{rcl}
d_1 & =& 2a^9 + 7a^8 - 5a^7 - 30 a^6 - 3 a^5 + 26 a^4 + 23 a^3 + 15 a^2 - 58a + 16,\\
d_2 & = & 2a^6 + 5a^5 - 8a^4 - 17a^3 + 16a^2 + 13a - 8.
  \end{array}$$
If $a\neq -\omega^{\pm 1}$, from $\gcd(d_1,d_2)=1$ we conclude that the matrix $M$ has rank $5$.
For $p=3$ and $a=-1$, we have $\det(M_{(1,2,3,4,6)})=1$;
for $p\neq 3$ and $a =  -\omega^{\pm 1}$, we have  $d_1\in \{2^6(\omega - 1), -2^6(\omega + 2)\}$, whence  $d_1\neq 0$.
So, also in these cases, $\rk(M)=5$.

Now, suppose that $\eta_{|U}$ does not have the eigenvalue $-1$. 
Then $\overline{U}$ contains $V_{-1}(\eta^T)$ and, in particular, contains the eigenvector 
$u=e_2+e_3$.
The matrix whose columns are the vectors 
$$u,\; y^T u, \;(y^2)^T u,\; (yx)^T u,\;  (yxy^2)^T u,\; ((yx)^2)^T u,\;(yxy)^T u,\; (y\eta)^T u $$
has determinant $2a(a+2)^4\neq 0$. We conclude that $\dim(\overline{U})=8$.
\end{proof}

\begin{corollary}\label{Om7}
If $p=2$, then $H$ is not contained in any maximal subgroup $\Sp_6(q)$ of class $\mathcal{S}$.
If $p$ is odd and $a\neq \pm 2$, then $H$ is neither contained in any maximal subgroup of class $\mathcal{C}_4$
nor in any maximal subgroup 
$2^\cdot\Omega_7(q)$ of class $\mathcal{S}$.
\end{corollary}

\begin{proof}
Suppose that $H$ is contained in  one of the maximal subgroups described in the statement.
By \cite[Table 8.50]{Ho} and \cite{K}, $\tau^i(M)$ belongs to the class $\mathcal{C}_1$ for some $i=1,2$. In particular,
$\tau^i(H)$ is reducible, in contrast with Lemmas \ref{irr+2}, \ref{irr+odd}, \ref{irr+tau} and \ref{irr+tau2}.
\end{proof}

\section{The $(2,3)$-generation of $\Omega_8^+(q)$}\label{gen8+}

\begin{lemma}\label{+C2}
Suppose $a\neq \pm 2$ if $p$ is odd.
Then $H$ is neither contained in any maximal subgroup of classes $\mathcal{C}_2, \mathcal{C}_6$
nor in any maximal subgroup $2^\cdot \Omega_8^+(2)$, $2^\cdot Sz(8)$, $2\times \Alt(10)$, $2^\cdot \Alt(10)$ of class $\mathcal{S}$.
\end{lemma}

\begin{proof} 
The characteristic polynomial of $xy$ is
$$\chi_{xy}(t)=\left\{\begin{array}{ll}
t^8 + (a^2 + a + 1)t^7 + a^2 t^6 + (a + 1)t^5 + (a + 1)^4 t^4+ \\
(a + 1)t^3 + a^2t^2 + (a^2 + a + 1)t + 1 & \textrm{if } p=2,\\
t^8 + (a + 1) t^7 - a t^6 - 2 t^5 + a^2 t^4 - 2 t^3 - a t^2 + (a + 1) t + 1    & \textrm{if } p>2.          
\end{array}\right. $$
Suppose that our claim is false. Note that $H$ is absolutely irreducible by Lemmas \ref{irr+2} and \ref{irr+odd}.\\
\textbf{Case 1.} $H\leq M$ with $M\in\mathcal{C}_2$ monomial, or $M$ one of the remaining maximal subgroups of the statement.\\
Then $q=p$ is odd and $\Upsilon(M) \subseteq\{2,3,5,7\}$. The element $\eta=(xy^2)^2(xy)^2$ of Lemma \ref{irr+odd} has order
divisible by $p$, whence  $q\in \{5,7\}$. So, assume $q=5$. If $a=1$,  the order of $\eta yxy$ is $2\cdot 31$;
if $a=-1$, the order of $\eta$ is $5\cdot 13$.
Finally, assume $q=7$. If $a=1$, then  $\eta y^2 x y$ has order $19$;  
if $a=\pm 3$, the element $y \eta^3$ has order $9\cdot 19$;
if $a=-1$, the element $y\eta^2 $ has order $43$.
In all these cases, we get a contradiction considering the order of $M$.
\smallskip

\noindent\textbf{Case 2.} $H\leq M \in \mathcal{C}_2$ fixes a decomposition 
$\F_q^8=\left\langle v_1,w_1\right\rangle\oplus \dots \oplus \left\langle v_4,w_4\right\rangle$.\\
Let $\nu: H \to \Sym(4)$ be the corresponding homomorphism and let $r$ be an odd prime dividing 
$|\ker(\nu)|$. Then $r$ divides $q \pm 1$. If $p$ is odd, consideration of $\eta$ as in Lemma \ref{irr+odd} 
gives $p = 3$.
We may assume that $y$ acts as $\zeta=(v_1,v_2,v_3)(w_1,w_2,w_3)$ and, considering its rational form, that $yv_4=v_4$, $yw_4=w_4$.
Hence, $x$ must act either as $\xi=\xi_1$ or as $\xi=\xi_2$, where for some  $A,B,C\in \GL_2(q)$:
$$\xi_1=\begin{pmatrix}
0&0&0&A\\
0&B&0&0\\
0&0&C&0\\
A^{-1}&0&0&0
\end{pmatrix},\quad
\xi_2=\begin{pmatrix}
0&0&0&A\\
0&0&B&0\\
0&B^{-1}&0&0\\
A^{-1}&0&0&0
\end{pmatrix}, \quad
B=\begin{pmatrix}
\beta_1&\beta_2\\
\beta_3&\beta_4
\end{pmatrix}.$$

Setting $\Delta=\det(B)$, the characteristic polynomial of $\xi\zeta$ is respectively:
$$\begin{array}{rcl}
p_1(t) & =&  t^8 + \alpha t^4 + \beta,\\ 
p_2(t) & =& 
t^8 - (\beta_1+\beta_4) t^7 + \Delta t^6 - \frac{\beta_1 + \beta_4}{\Delta} t^5 +\frac{
    (\beta_1 + \beta_4)^2}{\Delta} t^4 - (\beta_1 +\beta_4) t^3 + \\
    &&\frac{ 1}{\Delta} t^2 - 
\frac{\beta_1 + \beta_4}{\Delta} t + 1.
\end{array}$$
Comparison of $p_1(t)$ with $\chi_{xy}(t)$ gives immediately a contradiction.
Next we compare $p_2(t)$ with $\chi_{xy}(t)$. If $p=2$,
considering the coefficients of $t^6$ and of $t^2$ we get 
the contradiction $\Delta=a^2=\frac{1}{\Delta}$.
If $p=3$,  considering the coefficients of $t^7$ and of $t^3$ we get 
$a+1=-(\beta_1+\beta_4)=-2$, whence the contradiction $a=0$.
\smallskip

\noindent \textbf{Case 3.} $H\leq M \in \mathcal{C}_2$ fixes a decomposition 
$\F_q^8=V_1\oplus V_2$, with $\dim(V_i)=4$. \\
Clearly we must have $xV_1=V_2$ and $yV_i=V_i$, $i=1,2$.
It follows that $xy$ has trace $0$. 
This implies  $a=\omega^{\pm 1}$ if $p=2$, $a=-1$ if $p$ is odd.
If $p=2$ from $a=\omega^{\pm 1}$ we get the contradiction 
$0=\tr ((xy)^3)=\omega^{\mp 1}$. 
If $p$ is odd, from  $a=-1$ we get the contradiction
$0=\tr(y(xy)^3)=4$.
\end{proof}

\begin{lemma}\label{+C5}
The group $H$ is not contained in any maximal subgroup of class $\mathcal{C}_5$.
\end{lemma}

\begin{proof}
Suppose the contrary. Then there exists $g\in \GL_{8}(\F)$ such that
$x^g=\vartheta_1 x_0$, $y^g=\vartheta_2 y_0$, with $x_0,y_0\in \GL_{8}(q_0)$ and $\vartheta_i\in \F$.
If $p=2$, we have $\tr([x,y])=a^2$; if $p$ is odd, we have $\tr([x,y])= 2a$.
From  $\tr ([x,y])=\tr\left([x^g,y^g]\right)= \tr ([x_0,y_0])$ it follows that 
$a\in \F_{q_0}$. So,  $\F_q=\F_p[a]\leq \F_{q_0}$ implies $q_0=q$.
\end{proof}

\begin{lemma}\label{+C5S}
The group $H$ is not contained in any maximal subgroup of type $\Omega_8^-(\sqrt{q})$ of class~$\mathcal{S}$.
\end{lemma}

\begin{proof}
Suppose the contrary, and write $q=q_0^2$.
Then, for some $i=1,2$, there exists $g\in \GL_{8}(\F)$ such that
$(\tau^i(x))^g=\vartheta_1 x_0$, $(\tau^i(y))^g=\vartheta_2 y_0$, with $x_0,y_0\in \GL_{8}(q_0)$ and $\vartheta_i\in \F$.
First, assume $p=2$. We have $\tr([\tau(x),\tau(y)])=a^4$ and $\tr( \tau^2(x) \tau^2(y))=a^2$.
So, if $i=1$ we proceed as in Lemma \ref{+C5}. If $i=2$, then 
$a^2=\tr (\tau^2(x) \tau^2(y))=\tr\left((\tau^2(x))^g (\tau^2(y))^g \right)= \vartheta_1\vartheta_2 \tr (x_0y_0 )$.
It follows that $a\in \F_{q_0}$, whence the contradiction $\F_q=\F_2[a]\leq \F_{q_0}$.

A similar argument can be applied for $p$ odd. In this case, we have $\tr([\tau(x),\tau(y)])=-2a$ and 
$\tr( (\tau^2(x) \tau^2(y))^2)=2a + 1$.
\end{proof}

When $p\neq 3$, the class $\mathcal{S}$  contains maximal subgroups isomorphic either to $d\times \PSL_3(q).3$
or to $d\times \PSU_3(q^2).3$, where $d=\gcd(q-1,2)$.
The construction of these subgroups is made via the adjoint module $N$, described
in \cite[Section 5.4.1]{Ho} with the relevant properties. We may assume
$N=\{A\in \Mat_3(q)\mid \tr(A)=0\}$ for $G=\GL_3(q)$ and
$N=\{A\in \Mat_3(q^2) \mid A^T=A^\sigma,\; \tr(A)=0\}$ for
$G=\U_3(q^2)= \{g\in \GL_3(q^2) \mid g^T g^\sigma=\I_3\}$,
where $g^\sigma$ is obtained from $g$ raising its entries to $q$.
In both cases, $G$ normalizes $N$ and the conjugation action of $G$ on $N$ gives rise
to an absolutely irreducible representation
$\Phi: G'\to \Omega_8^+(q)$. We claim that $\Phi(g)$ admits the eigenvalue $1$ for each $g\in G'$,
equivalently that $C_{G'}(N)\neq \{0\}$. 
We will apply this argument to the elements $[x,y]$, $\tau([x,y])$ and $\tau^2([x,y])$.

\begin{remark}\label{poli}
The characteristic polynomial of $g_i=\tau^i([x,y])$ is:
$$\chi_{g_0}(t)=\left\{
\begin{array}{ll}
t^8 + a^2 t^7 + a^2 t^6 + a^2 (a+1)^2 t^5 + (a+ 1)^8 t^4 +\\
a^2 (a+1)^2 t^3 +   a^2 t^2 + a^2 t + 1 & \textrm{if } p=2; \\
t^8  -2 a t^7 + (3 a^2 - 4) t^6 -2 a (a^2-1) t^5  + (a^4 - 2a^2 + 6)t^4\\
-2 a  (a^2-1) t^3 + (3 a^2 - 4) t^2  -2 a t + 1 & \textrm{if } p>2;
  \end{array}\right.$$
$$\chi_{g_1}(t)=\left\{
\begin{array}{ll}
   t^8 + a^4 t^7 + a^2 t^6 +  a^2(a+1)^2 t^5 + (a+ 1)^4 t^4 + \\
   a^2 (a+1)^2  t^3 +  a^2 t^2 + a^4 t + 1 & \textrm{if } p=2; \\
   t^8 + 2 a t^7 + (3 a^2 - 4) t^6 + 2 a (a^2- 1) t^5 +
 (a^4 - 2 a^2 + 6) t^4  \\
 +2 a   (a^2-1) t^3 + (3 a^2 - 4) t^2 + 2 at + 1 & \textrm{if } p>2;
  \end{array}\right.$$
  $$\chi_{g_2}(t)=\left\{
\begin{array}{ll}
   t^8 + t^7 + a^2 t^6 + (a^3 + 1)^2 t^5 + a^4 (a+1)^4 t^4 +\\
   (a^3 + 1)^2 t^3 + a^2 t^2 + t + 1 & \textrm{if } p=2; \\
   t^8 + a^2 t^7 + (3 a^2 - 4) t^6 + 3 a^2 t^5 +  2(a^2 + 3)t^4 +  3 a^2 t^3 + \\
   (3 a^2 - 4) t^2 + a^2 t + 1 & \textrm{if } p>2.
  \end{array}\right.$$
In particular, for $p=2$ we have 
$$\chi_{g_0}(1)=(a+1)^8, \quad \chi_{g_1}(1)=(a+1)^4 \equad \chi_{g_2}(1)=a^4 (a+1)^4,$$
and for $p>2$ we have
$$\chi_{g_0}(1)=a^2(a-2)^2, \quad \chi_{g_1}(1)=a^2(a+2)^2 \equad \chi_{g_2}(1)=16 a^2.$$
\end{remark}

\begin{lemma}\label{LU3}
Suppose $a\neq \pm 2$ if $p$ is odd. Then  $H$ is neither contained in any maximal subgroup $d\times \PSL_3(q).3$ 
nor in any maximal subgroup $d\times \PSU_3(q^2).3$  of class $\mathcal{S}$.
\end{lemma}

\begin{proof}
Suppose first $g \in \SL_3(q)$. By Frobenius formula, the centralizer $C$ of $g$ in $\Mat_3(q)$
has dimension at least $3$. Hence, by Grassmann formula, the intersection of $C$ with $N$
has dimension at least $2$. If follows that $\Phi(g)$ has the eigenvalue $1$.
Considering $g \in \{[x,y], \tau([x,y]), \tau^2([x,y])\}$, by Remark \ref{poli} we get a contradiction. 

Suppose now $g \in \SU_3(q^2)$. The group $\SU_3(q^2)$ normalizes the subspace 
$$M=\left\{A\in \Mat_3(q^2)\mid A^T=A^{\sigma}\right\}$$ of dimension $9$ over $\F_q$.
The conjugation action gives rise to an absolutely irreducible $8$-dimensional representation $\Phi$ of $\SU_3(q^2)$ over 
$\frac{M}{\F_q^\ast \I_3}$.
We show that every $g\in \SU_3(q^2)$ centralizes some non-scalar matrix in $M$. Equivalently,  $\Phi(g)$ has the eigenvalue $1$.
Proceeding as in the previous case, this will give a contradiction.

We have $g^{-1}=(g^{\sigma})^{T}$, so  $g$ centralizes $g+(g^{\sigma})^{T}\in M$.
Hence, our claim is true unless $g+(g^{\sigma})^{T}=\rho \I_3$ for some $\rho\in \F_q$.
In this case, from $g(g^{\sigma})^{T}=\I_3$ it follows that $g^2-\rho g+\I_3=0$. Thus $g$ has minimal polynomial of degree at most $2$.
It follows that there exists a decomposition $(\F_{q^2})^3= \langle v_1\rangle \oplus \langle v_2, v_3\rangle$ fixed by $g$.

Suppose that  $v_1$ is singular. Since $v_1^\perp$ has dimension $2$, there exists 
$$u= \sum_i\lambda_i v_i\not\in \langle v_1\rangle$$
such that $v_1^Tu^{\sigma }=0$. It follows $v_1^T (\lambda_2v_2+\lambda_3v_3)^{\sigma}=0$.
In other words we may assume that $u=v_2$. Now, $v_2$ and $gv_2\in \langle v_2, v_3\rangle $ cannot be linearly independent otherwise,
setting $v_3=gv_2$, we would have $v_1^T v_1^{\sigma}=v_1^T v_2^{\sigma}= v_1^T v_3^{\sigma }$ and the form would be singular.
So, $v_2$ is non singular and $g\langle v_2\rangle = \langle v_2\rangle$. 
Substituting $v_1$ at the beginning with  $v_2$, we may assume $v_1$ non singular. It follows that
$\langle v_2, v_3\rangle =v_1^\perp$ and $g$ is conjugate, under $\U_3(q^2)$, to a matrix
of shape $\begin{pmatrix}
\alpha&0&0\\
0&\beta&\gamma\\
0&\delta&\epsilon
\end{pmatrix}$ centralized by $\begin{pmatrix}
1&0&0\\
0&0&0\\
0&0&0
\end{pmatrix}\in M$.
\end{proof}

\begin{lemma}\label{3D4}
The group $H$ is not contained in any maximal subgroup of type ${}^3D_4(\sqrt[3]{q})$ of class $\mathcal{S}$.
\end{lemma}

\begin{proof}
Let $q=q_0^3$ and assume, by sake of contradiction, that $H\leq d\times {}^3D_4(q_0)$, where $d=\gcd(q-1,2)$.
Then, there exists a non singular matrix $P$ and a scalar $\lambda_h$ such that, for some $i=1,2$:
$$P^{-1} (\tau^i( h^{\varphi}) ) P = \pm h \quad \textrm{ for all  } h \in H,$$
where $h^\varphi$ denotes the matrix obtained by substituting each entry $\alpha$ of $h$ with $\alpha^{q_0}$.
In particular, taking $h=[x,y]$, $\tau^i(h^\varphi)$ and $h$ have the same characteristic polynomial.

The characteristic polynomials of $[x,y]$, $\tau([x,y])$ and $\tau^2([x,y])$ are described in Remark \ref{poli}.
If $p=2$, equating the coefficients of the term $t^6$ in $\chi_{[x,y]}(t)$  and in  $\chi_{\tau^i([x,y]^\varphi)}(t)$,
we get the contradiction $a^2=(a^2)^{q_0}$.
Similarly, for $p$ odd, equating the coefficients of the term $t^7$ in $\chi_{[x,y]}(t)$  and in  $\chi_{\tau([x,y]^\varphi)}(t)$,
we get $a^{q_0}=-a$ which gives the contradiction $a^{q}=-a$.
Finally, equating the coefficients of the terms $t^7$ and $t^5$ in $\chi_{[x,y]}(t)$  and in  $\chi_{\tau^2([x,y]^\varphi)}(t)$,
we get $a^{2q_0}=-2a$ and $3 a^{2 q_0}=-2a(a^2-1)$ which gives the contradiction  $a=-2$.
\end{proof}

\begin{theorem}\label{main+}
Suppose $q\geq 4$. Let $a \in \F_q^*$ be such that $\F_q=\F_p[a]$. If $p$ is odd, assume also that $a\neq \pm 2$.
Then $H=\Omega_8^+(q)$. In particular, $\Omega_8^+(q)$ is $(2,3)$-generated for
all $q\geq 4$.
\end{theorem}

\begin{proof}
By the considerations of Section \ref{gen+}, Lemma \ref{irr+2} and  Lemma \ref{irr+odd},  $H$ is an absolutely irreducible
subgroup of $\Omega_8^+(q)$. By Lemmas \ref{+C2} and \ref{+C5} and by Corollary \ref{Om7}, it is primitive and tensor-indecomposable,
and is not contained in any maximal subgroup of classes $\mathcal{C}_5, \mathcal{C}_6$.
So, either $H=\Omega_8^+(q)$ or  $H$ is contained in one of the following  maximal subgroups of class $\mathcal{S}$
(here $d=\gcd(q-1,2)$):
\begin{itemize}
\item[(i)] $\Sp_6(q)$ if $p=2$, or $2^\cdot \Omega_7(q)$ if $p$ is odd; 
\item[(ii)] $d^\cdot\Omega_8^-(q_0)$, where $q=q_0^2$;
\item[(iii)] $d\times \PSL_3(q).3$ if $q\equiv 1 \pmod 3$, or  $d \times\PSU_3(q^2).3$ if $q\equiv 2 \pmod 3$;
\item[(iv)] $d\times {}^3D_4(q_0)$, where $q=q_0^3$;
\item[(v)] $2^\cdot \Omega_8^+(2)$ if $q=p$ is odd;
\item[(vi)] $2^\cdot Sz(8)$, $2\times \Alt(10)$ or $2^\cdot \Alt(10)$ if $q=5$.
\end{itemize}
Now, case (i) is excluded by Corollary \ref{Om7};
cases (ii), (iii) and (iv) are excluded, respectively, by Lemmas \ref{+C5S}, \ref{LU3} and \ref{3D4};
cases (v) and (vi) are excluded by Lemma \ref{+C2}.
\end{proof}

\section{Generators for $\Omega_8^-(q)$}

\subsection{Fields of characteristic $2$}\label{sec2}

We first deal with the case $q=2$.

\begin{lemma}\label{q2}
The group $\Omega_8^-(2)$ is $(2,3)$-generated. 
\end{lemma}

\begin{proof}
Take $x,y$ as follows:
$$x=\begin{pmatrix}
0 & 1 &  0 & 0 &  1 & 0 & 0 & 0 \\
1 & 1 &  0 & 0 & 0 & 1 &  0 & 0 \\
0 & 0 &  0 & 1 & 0 & 0 & 0 & 0 \\
0 & 0 &  1 & 0 &  0 & 0 &  0 & 0 \\
0 & 1 &  0 & 0 &  0 & 1 &  0 & 0 \\
1 & 1 &  0 & 0 &  1 & 1 & 0 & 0 \\
0 & 0 &  0 & 0 &  0 & 0 & 0 & 1 \\
0 & 0 &  0 & 0 &  0 & 0 & 1 & 0
  \end{pmatrix},\quad y=\begin{pmatrix}
 0 & 0 & 0 & 0 & 1 & 0 & 0 & 0\\
 0 & 0 & 0 & 1 & 0 & 0 & 0 & 0\\
 0 & 1 & 0 & 0 & 0 & 0 & 0 & 0\\
 0 & 0 & 1 & 0 & 0 & 0 & 0 & 0\\
 1 & 0 & 0 & 0 & 1 & 0 & 0 & 0\\
 0 & 0 & 0 & 0 & 0 & 0 & 0 & 1\\
 0 & 0 & 0 & 0 & 0 & 1 & 0 & 0\\
 0 & 0 & 0 & 0 & 0 & 0 & 1 & 0
\end{pmatrix}.$$
Then, $x$ is an involution, $y$ has order $3$ and $\det(x)=\det(y)=1$.
Both the elements $x$ and $y$ fix the quadratic form
$$Q\left(\sum_{i=1}^4 (\alpha_i e_i+\alpha_{-i} e_{-i}) \right)=\alpha_1^2+\alpha_{-1}^2+\sum_{i=1}^4 \alpha_i\alpha_{-i}.$$
Since $\rk(\I_{8}-x) = 4$, by \cite[Proposition 1.6.11]{Ho}(i), the quasideterminant of $x$ is $+1$,
and hence $x\in \Omega_8^{\epsilon}(q)$.
The space $V=\F_2^8$ contains a totally singular subspace $W=\langle e_2,e_3,e_4  \rangle$ of dimension $3$.
Suppose there exists $w=\mu_1 e_1 + \sum\limits_{i=1}^4 \mu_{-i} e_{-i}$ such that
$\langle W, w\rangle$ is  totally singular of dimension $4$.
Imposing that $w$ is orthogonal to $W$, we obtain $\mu_{-2}=\mu_{-3}=\mu_{-4}=0$.
From $Q(w)=0$ we get $\mu_1^2 + \mu_1\mu_{-1} + \mu_{-1}^2=0$.
In particular, $\mu_{-1}\neq 0$ and hence we may assume $\mu_{-1}=1$.
Then, we get the condition $\mu_1^2+\mu_1+1=0$. Since the polynomial $t^2+t+1$ is irreducible over $\F_2$, 
$W$ is not contained in a totally singular subspace of dimension $4$. By Witt's lemma, the Witt index of $V$ is $3$ and
so $Q$ is a quadratic form of sign $-$. This means that $H=\langle x,y\rangle\leq \Omega_8^-(2)$.
Now, 
$$\bigcup_{j=3}^5 \Upsilon(  [x,y]^j xyx) =\Upsilon(\Omega_{8}^-(2)).$$
By \cite{Atlas} we conclude that  $H=\Omega_{8}^-(2)$.
\end{proof}

Assume $q\geq 4$ even, and let $a \in \F_q^*$ be such that $\F_2[a]=\F_q$.
This means, in particular, that $a\neq 0,1$. Define the following elements $x=x(a)$ and $y$:
$$x=\begin{pmatrix}
0 & a^{-1} &  0 & 0 & 0 & 0 &  0 & 0 \\
a &  0 &  0 & 0 &  0 & 0 &  0 & 0 \\
0 & 0 & \frac{a^2}{(a+1)^2} &  \frac{1}{(a+1)^2} &  0 & 0 &  0 & 0 \\
0 & 0 &  \frac{1}{(a+1)^2} &  \frac{a^2}{(a+1)^2} &  0 & 0 &  0 & 0 \\
(a+1)^2 &  a^{-1} &  0 & 0 &  0 & a & 0 & 0 \\
a^{-1} & \frac{(a+1)^2}{a^2} & 0 & 0 & a^{-1} & 0 &  0 & 0 \\
0 & 0 &  0  & 0 &  0 & 0& \frac{a^2}{(a+1)^2} & \frac{1}{(a+1)^2} \\
0 & 0 &  0 & 0 &  0 & 0 &  \frac{1}{(a+1)^2} & \frac{a^2}{(a+1)^2}
  \end{pmatrix},$$
  $$y=\begin{pmatrix}
 0 & 0 & 0 & 0 & 1 & 0 & 0 & 0\\
 0 & 0 & 0 & 1 & 0 & 0 & 0 & 0\\
 0 & 1 & 0 & 0 & 0 & 0 & 0 & 0\\
 0 & 0 & 1 & 0 & 0 & 0 & 0 & 0\\
 1 & 0 & 0 & 0 & 1 & 0 & 0 & 0\\
 0 & 0 & 0 & 0 & 0 & 0 & 0 & 1\\
 0 & 0 & 0 & 0 & 0 & 1 & 0 & 0\\
 0 & 0 & 0 & 0 & 0 & 0 & 1 & 0
\end{pmatrix}.$$
Then, $x$ is an involution, $y$ has order $3$ and $\det(x)=\det(y)=1$.
Both the elements $x$ and $y$ fix the quadratic form
$$Q\left(\sum_{i=1}^4 (\alpha_i e_i+\alpha_{-i} e_{-i})\right)=\alpha_1^2+\alpha_{-1}^2+
a^2 \sum_{i=2}^4 (\alpha_i^2+\alpha_{-i}^2)+\sum_{i=1}^4 \alpha_i\alpha_{-i},$$
whose associated symmetric form has Gram matrix $J=\begin{pmatrix}
  0 & \I_4 \\ \I_4 & 0   
    \end{pmatrix}$.
    
Clearly, $y \in \Omega_8^{\epsilon}(q)$, having order $3$. 
Since $\rk(\I_{8}-x) = 4$, by \cite[Proposition 1.6.11]{Ho}(i), the quasideterminant of $x$ is $+1$,
hence $x\in \Omega_8^{\epsilon}(q)$.
The space $V=\F_q^8$ contains a totally singular subspace $W=\langle ae_1+e_2, ae_1+e_3,ae_1+e_4  \rangle$ of dimension $3$.
Suppose that there exists a vector $w=\mu_1 e_1 + \sum\limits_{i=1}^4 \mu_{-i} e_{-i}$ such that 
$\langle W, w\rangle$ is totally singular  of dimension $4$.
Imposing that $w$ is orthogonal to $W$, we obtain $\mu_{-2}=\mu_{-3}=\mu_{-4}=a\mu_{-1}$. 
Now, $Q(w)=0$ gives $\mu_{-1}\neq 0$.  Hence, we may assume $\mu_{-1}=1$, obtaining the necessary condition $\mu_1^2+\mu_1+(a+1)^4=0$.
So, if the polynomial $t^2+t+(a+1)^4$ is irreducible in $\F_q[t]$, then
$W$ is not contained in a totally singular subspace of dimension $4$. By Witt's lemma \cite[Corollary 2.1.7]{KL},
the Witt index of $V$ is $3$ and so $Q$ is a quadratic form of sign $-$.
Thus,
$$H=\langle x, y\rangle\leq \Omega_8^-(q) \quad  \textrm{if } t^2+t+(a+1)^4 \textrm{ is irreducible over } \F_q.$$

\begin{lemma}\label{pol_irr}
For all $q=2^f$, there exists $a\in \F_q$ satisfying $\F_q=\F_2[a]$
and such that the polynomial $p(t)=t^2+t+(a+1)^4$ is irreducible over $\F_q$. 
\end{lemma}

\begin{proof}
By Carlitz's formula \cite{Car}, there exist $2^{f-1}$ elements $\alpha$ in $\F_q$ such that
$t^2+t+\alpha$ is irreducible.
We claim that $\F_2[\alpha]=\F_q$ for at least one $\alpha$.
Indeed, the number of elements lying in proper subfields of $\F_q$ is $0$ if $f=1$, it is
$2$ if $f$ is a prime, and it is at most
 $2+2^2+ \dots +2^{f-2}=2^{f-1}-2 < 2^{f-1}$ otherwise. Now, write $\alpha=(a+1)^4$:
clearly, $\F_2[a]=\F_2[a+1]=\F_2[\alpha]=\F_q$.
\end{proof}

The characteristic polynomial of $xy$ is
\begin{equation}\label{char1}
\chi_{xy}(t)=(t + (a^2 + 1)^{-1}) (t + a^2 + 1) (t^2 + (a^2 + 1)^{-1} t + 1)(t^2 + t + 1)^2.
\end{equation}
Note that $\tr(xy)=(a+1)^2$ and that $\chi_{xy}(1)=\frac{a^4}{(a+1)^4}\neq 0$.

\begin{lemma}\label{irr_beta}
The group $H$ is absolutely irreducible. 
\end{lemma}

\begin{proof}
Let $U$ be an $H$-invariant subspace of $\F^{8}$ and $\overline{U}$ be an $H^T$-invariant  complement.
Assume $a^2+a+1\neq 0$.  In this case, the characteristic polynomial of $xy$ admits $(a+1)^2$ as a simple root.
The corresponding eigenspace of $xy$ and $(xy)^T$ are generated, respectively, by the vector $s=a^2 e_2+ e_3+ (a^3+a) e_{-1}$ and 
$\overline{s}=e_3+a^2e_4+(a^5+a^3)e_{-1}+a^2e_{-3}+a^4e_{-4}$.
When $a^2+a+1=0$, the eigenspace of $xy$ and $(xy)^T$ relative to $(a+1)^2$ are still one-dimensional and generated by the same vectors $s$ 
and $\overline{s}$.

If the restriction $(xy)_{|U}$ has the eigenvalue $(a+1)^2$, then the vector $s$ belongs to $U$.
Take the  matrix $M$ whose columns are the images of $s$ by
$$\I_8,\; y,\; y^2, \; xy^2, \; yxy^2,\; xyxy^2, \; (xy^2)^2, \; (xy^2)^3.$$
Since $\det(M)=\frac{(a+1)^{18}}{a^8}\neq 0$, then $U=\F^{8}$.
If the restriction $(xy)_{|U}$ does not have the eigenvalue $(a+1)^2$, then $\overline U$ contains the vector $\overline{s}$.
Take the matrix $\overline{M}$  whose columns are the images of $\overline{s}$ by
$$\I_8,\; y^T, \; (y^2)^T,\; (yxy^2)^T, \; (xy^2xy)^T, \;  ((xy^2)^3)^T, \; (y(xy^2)^3)^T, \; (y(xy^2)^4)^T.$$
Then $\det(\overline{M})=a^8 (a+1)^{22}\neq 0$, and hence $\overline{U}=\F^8$, that is, $U=\{0\}$.
\end{proof}

\subsection{Fields of odd characteristic}\label{sec2odd}

Suppose $q$ odd.
Let $a,\xi\in \F_q^*$, where $\F_p[a]=\F_q$ and $\xi$ is a nonsquare in $\F_q^*$.
Define the following elements $x=x(a,\xi)$ and $y$:
$$x=\begin{pmatrix}
0 & 0 &  0 & 0 & 2 & 0 &  0 & 0 \\
0 & 1 &  0 & 0 &  0 & 0 &  0 & 0 \\
0 & 0 & -1 & 0 &  0 & 0 &  -\frac{a^2\xi}{2} & -a\xi\\
0 & 1 &  0 &  -1 &  0 & 0 &  0 & 0 \\
\frac{1}{2} & 0 &  0 & 0 &  0 & 0 & 0 & 0 \\
0 & -\frac{1}{2} & 0 & 1 & 0 & 1 &  0 & 0 \\
0 & 0 &  0  & 0 &  0 & 0& -1 & 0 \\
0 & 0 &  0 & 0 &  0 & 0 &  a & 1
  \end{pmatrix},\quad
  y=\begin{pmatrix}
 0 & 0 & 1 & 0 & 0 & 0 & 0 & 0\\
 1 & 0 & 0 & 0 & 0 & 0 & 0 & 0\\
 0 & 1 & 0 & 0 & 0 & 0 & 0 & 0\\
 0 & 0 & 0 & 1 & 0 & 0 & 0 & 0\\
 0 & 0 & 0 & 0 & 0 & 0 & 1 & 0\\
 0 & 0 & 0 & 0 & 1 & 0 & 0 & 0\\
 0 & 0 & 0 & 0 & 0 & 1 & 0 & 0\\
 0 & 0 & 0 & 0 & 0 & 0 & 0 & 1
\end{pmatrix}.$$
Then, $x$ is an involution, $y$ has order $3$ and $\det(x)=\det(y)=1$.
Both the elements $x$ and $y$ fix the symmetric form whose Gram matrix is 
$$J=\begin{pmatrix}
0 & 0 & \I_3 & 0\\
0 & 1 & 0 & 0 \\
\I_3 & 0 & 0 & 0 \\
0 & 0 & 0 & -\xi
    \end{pmatrix} $$
of determinant $\xi$.
By \cite[Proposition 1.5.42]{Ho} the subgroup $H=\langle x,y\rangle$ is contained in 
$\SO_8^-(q)$ as $\xi$ is a nonsquare.
Furthermore, $y \in \Omega_{8}^{-}(q)$ having order $3$. 
The spinor norm of $x$ with respect to $J$ is $+1$:
if follows from \cite[Proposition 1.6.11]{Ho}(ii) that $x$ belongs to $\Omega_{8}^-(q)$. Hence,
$$H=\langle x,y\rangle\leq \Omega_8^-(q).$$

The characteristic polynomial of $xy$ is
\begin{equation}\label{charxy}
\chi_{xy}(t)= t^8 -t^6 +  \frac{a^2\xi- 4}{4} t^5 + \frac{a^2\xi+ 4}{2} t^4 + \frac{a^2\xi- 4}{4} t^3 - t^2 + 1.
\end{equation}
Note that $\chi_{xy}(1)=a^2 \xi \neq 0$.
 
\begin{lemma}\label{-irrodd}
The group $H$ is absolutely irreducible.
\end{lemma}

\begin{proof}
By direct computations,  $E_\omega(y)$ is generated by the vectors
$v_i=v_i(\omega)$, $i=1,2$, where:
$$v_1=e_1+\omega^{-1} e_2+\omega e_3 \equad v_2=Jv_1=e_{-1}+\omega^{-1} e_{-2}+\omega e_{-3}.$$
Let $U$ be an $H$-invariant subspace of $\F^8$ and $\overline U$ be an $H^T$-invariant complement of $U$. 
We may assume $\dim U\leq 4$. 

If $U\cap E_\omega(y)=\{0\}$, then $E_\omega(y^T)=E_\omega(y^{-1})=E_{\omega^{-1}}(y) \leq \overline{U}$.
Writing $v_1^-=v_1(\omega^{-1})$ and $v_2^-=v_2(\omega^{-1})$, the matrix whose columns are the vectors 
$$v_1^-,\; x^Tv_1^-,\; (xy)^Tv_1^-,\; (xyx)^T v_1^-,\; ((xy)^2)^T v_1^-,\; ((xy)^2x)^T v_1^-,\;v_2^-,\; x^Tv_2^-$$
has determinant $\frac{1}{32} a\xi (a^2\xi-4)^2 (a^2\xi-16\omega^2)^2\neq 0$. This means that $U=\{0\}$.

So, for some $(b_1,b_2)\neq (0,0)$, we have $v=b_1 v_1 + b_2 v_2 \in U$.
Let $M$ be the matrix whose columns are 
$$v,\; xv,\; yx v,\; xy^2x v,\; (yx)^2 v.$$
We will show that $\rk(M)=5$, a contradiction as we are assuming $\dim(U)\leq 4$. 

Suppose first $b_2=0$, $b_1=1$. We have $\det(M_{(1, 2, 5, 7, 8)})=-\frac{a}{32}(a^2\xi-4)\neq 0$.
Now, suppose $b_2=1$. 
If $p=3$,  then 
$\det(M_{(1,2,3,4,8)})= -a b_1 (a^2\xi-1)^3$, which is nonzero
unless $b_1=0$. In this case, 
$\det(M_{(3, 4, 6, 7, 8 ) })= a^5\xi^2 (a^2\xi-1)\neq 0$.

If $p\neq 3$, then $\det(M_{(4, 5, 6, 7, 8)})=
\frac{2\omega + 1}{16}  a(a^2\xi-4)   (b_1+2\omega) (4(\omega + 2) b_1+ a^2\xi - 4)$
which is nonzero unless either (i) $b_1=-2\omega$  or (ii) $b_1=\frac{a^2\xi-4}{4(\omega^2-1)}$.
In case (i), we have $\det(M_{(2, 5, 6, 7, 8 )})= 2(\omega+2) a (a^2\xi-4)\neq 0$.
In case (ii), we have $\det(M_{( 1, 2, 3, 4, 8)})=\frac{\omega - 1}{2^6\cdot 3} a (a^2\xi-4)^2 (a^2\xi - 8\omega + 4)^2$,
which is nonzero unless $\xi=4a^{-2}(2\omega-1)$. This gives $\det(M_{( 2,5,6,7,8})=-48a \neq 0$.
\end{proof}

\begin{lemma}\label{-C5odd}
Suppose that $\F_p[a^2\xi]=\F_q$. Then,  $H$ is not contained in any maximal subgroup of class $\mathcal{C}_5$.
\end{lemma}

\begin{proof}
Suppose the contrary. Then there exists $g\in \GL_{8}(\F)$ such that
$x^g=\vartheta_1 x_0$, $y^g=\vartheta_2 y_0$, with $x_0,y_0\in \GL_{8}(q_0)$ and $\vartheta_i\in \F$.
From  $-2a^2\xi - 6=\tr ((xy)^4)=\tr\left((x^gy^g)^4\right)= \vartheta_1\vartheta_2 \tr ( (x_0y_0)^4)$,
it follows that 
$a^2\xi \in \F_{q_0}$. So,  $\F_q=\F_p[a^2\xi]\leq \F_{q_0}$ implies $q_0=q$.
\end{proof}

\section{The $(2,3)$-generation of $\Omega_8^-(q)$}

\begin{lemma}\label{S-}
The group $H$ is not contained in any maximal subgroup of class $\mathcal{S}$.
\end{lemma}

\begin{proof}
The class $\mathcal{S}$ is nonempty only when $p\neq 3$, in which case it contains 
only maximal subgroups isomorphic either to $\PSL_3(q)$
or to $\PSU_3(q^2)$. 
Arguing as in Lemma \ref{LU3}, we can exclude these subgroups since $xy$ does not have the eigenvalue $1$,
see  \eqref{char1} and \eqref{charxy}.
\end{proof}

\begin{theorem}\label{main-2}
Let $q\geq 4$ be even, and let $a \in \F_q$ be such that $\F_2[a]=\F_q$.
If the polynomial $t^2+t+(a+1)^4$ is irreducible, then $H=\Omega_8^-(q)$. In particular,  $\Omega_8^-(q)$ is $(2,3)$-generated for
all even $q\geq 2$.
\end{theorem}

\begin{proof}
By the considerations of Section \ref{sec2} and Lemma \ref{irr_beta}, $H$ is an absolutely irreducible subgroup of $\Omega_8^-(q)$.
If our claim is false, there exists a maximal subgroup $M$ of $\Omega_8^-(q)$ which contains $H$.
Then, $M$ must belong either to the class $\mathcal{C}_5$ or to the class $\mathcal{S}$.
The second possibility is excluded by Lemma \ref{S-}.
Hence, $M=\Omega_8^-(q_0)$ for some $q_0$ such that $q=q_0^r$ with $r$ an odd prime.
However, also this possibility can be easily excluded, as $\tr(xy)=a^2+1$ and $\F_2[a]=\F_q$.

For the second part of the statement, if $q=2$ we apply Lemma \ref{q2}.
If $q\geq 4$, there exists an element $a\in \F_q^*$ satisfying all the requirements  by Lemma \ref{pol_irr}.
\end{proof}

\begin{theorem}\label{main-odd}
Let $q\geq 3$ be odd.
Let $\xi$ be a nonsquare in $\F_q^*$ and let $a \in \F_q^*$ be such that $\F_p[a^2\xi]=\F_q$.
Then $H=\Omega_8^-(q)$. In particular, $\Omega_8^-(q)$ is $(2,3)$-generated for
all odd $q\geq 3$.
\end{theorem}

\begin{proof}
Take $a,\xi$ as in the statement (for instance, $a=1$ and any generator $\xi$ of the 
cyclic group $\F_q^*$).
By the considerations of Section \ref{sec2odd} and Lemma \ref{-irrodd}, $H$ is an absolutely irreducible subgroup of
$\Omega_8^-(q)$.
If our claim is false, there exists a maximal subgroup $M$ of $\Omega_8^-(q)$ which contains $H$.
So, $M$ must belong either to the class $\mathcal{C}_5$ or to the class $\mathcal{S}$.
However, these possibilities are excluded by Lemma  \ref{-C5odd}  and Lemma \ref{S-}, respectively.
\end{proof}

\end{document}